\documentclass[psamsfonts,reqno]{amsart}

\usepackage{amsmath}
\usepackage{amsfonts}
\usepackage{amssymb}
\usepackage{amssymb,amsfonts,calligra,verbatim,hyperref}
\usepackage[all,arc]{xy}
\usepackage{enumerate}
\usepackage{mathrsfs}
\usepackage{tikz-cd}
\usepackage{mathtools}
\usepackage{comment}

\usepackage{color}   
\usepackage{hyperref}
\hypersetup{
    colorlinks=true, 
    linktoc=all,     
    linkcolor=blue,  
}

\newtheorem{thm}{Theorem}[section]

\newtheorem{prop}[thm]{Proposition}
\newtheorem{lem}[thm]{Lemma}

\theoremstyle{definition}
\newtheorem{defn}[thm]{Definition}

\theoremstyle{remark}
\newtheorem{rem}[thm]{Remark}

\newcommand{\spec}{\operatorname{Spec}}
\DeclareMathOperator{\Hom}{\mathscr{H}\text{\kern -3pt {\calligra\large om}}\,}

\begin{document}

\makeatletter
\let\c@equation\c@thm
\makeatother
\numberwithin{equation}{section}

\title{Abelian Log Fundamental Group scheme}

\author{Aritra Sen}

\date{}

\begin{abstract}
Let $S$ be a connected Dedekind scheme and $X$ be  a proper smooth connected scheme over $S$ . Let $D$ a divisor with no multiplicity of $X$ such that the irreducible components of $D$ and as well their intersections are smooth over $S$. Now if we endow $X$ with the log structure associated with $D$ then the structure morphism from $X$ to $S$ is log-smooth. Let $x: S \to X$ be a $S$-point such that it doesn't intersect $D$. Then we prove that the maximal abelian quotient of the log Nori fundamental group scheme of $X$ fits in to an exact sequence of the form $0 \rightarrow (\mathbf{NS}^{\tau}_{X/S,D})^{\vee} \rightarrow (\pi^\text{log}_{\text{Nori}}(X,x))^{\text{ab}} \rightarrow \underset{n}{\varprojlim} \mathbf{Alb}_{X/S,D}[n] \rightarrow 0$. Here $\mathbf{NS}^{\tau}_{X/S,D}$ is the   torsion subgroup scheme of the generalized Neron-Severi group and $\mathbf{Alb}_{X/S,D}$ is the generalized Albanese scheme associated with the divisor $D$.
\end{abstract}

\maketitle

\tableofcontents

\section{Introduction}

In \cite{antei2011abelian} the following result was proved
\begin{thm}
Let $S$ be a connected Dedekind scheme and $X$ be a proper smooth connected $S$-scheme with a $S$-point $x: S \to X$. Then the maximally abelian quotient of the Nori fundamental scheme fits in an exact sequence of group schemes $0 \rightarrow (\mathbf{NS}^{\tau}_{X/S})^{\vee} \rightarrow (\pi_{\text{Nori}}(X,x))^{\text{ab}} \rightarrow \underset{n}{\varprojlim} \mathbf{Alb}_{X/S}[n] \rightarrow 0$.
\end{thm}

The Neron-Severi group scheme $\mathbf{NS}_{X/S}$ of $X$ is defined  as the group scheme $\mathbf{Pic}_{X/S}/\mathbf{Pic}^{0}_{X/S}$. The torsion subgroup scheme of $\mathbf{NS}_{X/S}$ is denoted by $\mathbf{NS}^{\tau}_{X/S}$.
The Albanese scheme $\mathbf{Alb}_{X/S}$ is the scheme which  represents the functor $\underline{\mathbf{Ext}}^{1}(\mathbf{Pic}^{0}_{X/S}, \mathbb{G}_m)$.

In this paper we prove a log scheme theoretic version of the above theorem.
Let $S$ be a connected Dedekind Scheme and $X= (X,\mathcal{O}_X)$ a  proper smooth connected scheme over $S$. Let $D$ be a divisor with no multiplicity of $X$ such that the irreducible components of $D$ as well as their intersections are smooth over $S$. Now we endow $X$ with the log structure associated with $D$. Then the structure morphism $f$ from $X$ to $S$ is log-smooth. Let $x$ be a $S$-point $x: S \to X$ such that it doesn't intersect $D$. Then we have the following theorem
\begin{thm}
\label{mainthm}
The maximal abelian quotient of the log Nori fundamental group scheme which is a profinite group scheme over $S$ fits in to an exact sequence of the form $$0 \rightarrow (\mathbf{NS}^{\tau}_{X/S,D})^{\vee} \rightarrow (\pi^\text{log}_{\text{Nori}}(X,x))^{\text{ab}} \rightarrow \underset{n}{\varprojlim} \mathbf{Alb}_{X/S,D}[n] \rightarrow 0.$$ Here $\mathbf{NS}^{\tau}_{X/S,D}$ is the torsion subgroup scheme of the generalized Neron-Severi group and $ \mathbf{Alb}_{X/S,D}$ is the generalized Albanese scheme associated with the divisor $D$.
\end{thm}

Let  $D_i$ be the irreducible components of $D$. Consider the constant sheaf $\mathbb{Z}$ on the flat topology of $D_i$ and there is a map $f:D_i \to S$. We define  $\mathbb{Z}_{D_i}$ to be $f_{*}\mathbb{Z}$ where the direct image is taken on the flat topology. Then we have a homomorphism from $((\bigoplus_{1 \leq i \leq m} \mathbb{Z}_{D_i})$ to $\mathbf{NS}_{X/S}$ . We define the generalized Neron-Severi group $\mathbf{NS}_{X/S,D}$ as $\operatorname{coker}[((\bigoplus_{1 \leq i \leq m} \mathbb{Z}_{D_i}) \to \mathbf{NS}_{X/S} ] $ and $\mathbf{NS}^{\tau}_{X/S,D}$ is the torsion subgroup scheme of the group scheme $\mathbf{NS}_{X/S,D}$.
 
 We define $(\bigoplus_{1 \leq i \leq m} \mathbb{Z}_{D_i})^{0} $ as  $\operatorname{ker}[((\bigoplus_{1 \leq i \leq m} \mathbb{Z}_{D_i}) \to NS_{X/S} ] $. This give us a complex $C_{0} = [(\bigoplus_{1 \leq i \leq m} \mathbb{Z}_{D_i})^{0} \rightarrow \mathbf{Pic}^{0}_{X/S}]$ where $\mathbf{Pic}^{0}_{X/S}$ is the degree $0$ part. Then 
$\mathbf{Alb}_{X/S,D}$ is the scheme representing the 
functor $\underline{\mathbf{Ext}}^{1}(C_{0}, \mathbb{G}_m)$. $\mathbf{Alb}_{X/S,D}$ fits into an exact sequence of the form

$$0 \to \mathbf{Hom}((\bigoplus_{1 \leq i \leq m} \mathbb{Z}_{D_i})^{0}, \mathbb{G}_m) \to \mathbf{Alb}_{X/S,D} \to \mathbf{Alb}_{X/S} \to 0 .$$
\textit{Acknowledgement}. I would like to express my gratitude to Professor Kazuya Kato for his kind advice and feedback while writing this article without which this paper would have never been completed. I would like to thank Professor Chikara Nakayama for his help  especially for completing section 6 of this paper. I would also like to thank Professor Madhav Nori for his advice and answering several questions. 
\section{Nori fundamental group scheme} \label{sec2}
Let $X$ be a scheme over a Dedekind scheme $S$. Let $x$ be a $S$-point of $x: S \to X$. Consider the category $\mathcal{T}(X,x)$ whose objects consists of triples of the form $(G,T,t)$ where $G$ is a finite flat group scheme over $S$, $T$ is a $G$-torsor in the fqpc topology of $X$, and $t$ is a $S$-rational point which lies above $x$. A morphism in $\mathcal{T}(X,x)$ between $(G_1,T_1,t_1)$ to $(G_2,T_2,t_2)$ consists of a pair $g,f$ where $g$ is a group scheme homomorphism over $S$ and $f$ is a morphism of schemes from $T_1$ to $T_2$ over $X$ such that the following diagram commutes:

\begin{center}
\begin{tikzcd}
G_1 \times T_1 \arrow{r}{m_1} \arrow[swap]{d}{g \times f} &  T_1 \arrow{d}{f} \\
G_2 \times T_2 \arrow{r}{m_2} & T_2
\end{tikzcd}
\end{center}

where $m_1:G_1 \times T_1 \to T_1$ and $m_2:G_2 \times T_2 \to T_2$ is the action of $G_1$ on $T_1$ and action of $G_2$ on $T_2$ respectively.

The Nori fundamental group scheme is defined as follows

\begin{defn}
A profinite  flat group scheme $\pi_{\text{Nori}}(X,x)$ is called Nori fundamental group scheme if there is a triple $(\pi_{\text{Nori}}(X,x),T,t)$ such that $T$ is a $\pi_{\text{Nori}}(X,x)$-torsor over 
$X$ and $t$ is a $S$-rational point over $x$ and there exists a unique morphism from $
(\pi_{\text{Nori}}(X,x),T,t)$ to any object $(G',T',t')$ in $\mathcal{T}(X,x)$.
\end{defn}

If the category $\mathcal{T}(X,x)$ is cofiltered then  $
\pi_{\text{Nori}}(X,x)$ exists and one can define it as follows

$$ \pi^{\text{Nori}}_{log}(X) = \lim_{(G,T,\alpha) \in Ob(\mathcal{T}(X,x))} G.$$

Here limit is the projective limit of finite flat group schemes.  To prove that the category $\mathcal{T}(X,x)$ is cofiltered, it is enough to show that the category $\mathcal{T}(X,x)$ has fiber products. 

Now consider the following two conditions

\begin{itemize}
    \item (P1) $X \to S$ is locally of finite-type, separated, faithfully flat, and for all points $s \in S$ S, $X_s$ is reduced;
    \item (P2) $X \to S$ is locally of finite-type, separated, faithfully flat, integral and normal;
\end{itemize}

The following proposition was proved in \cite{gasbarri2020existence} (Theorem 4.2).
\begin{thm}
\label{torsorsfibreproductclassical}
Let $X$ be a connected scheme, locally finite and faithfully flat on S. Let $x \in X(S)$ be a point. If one of the conditions (P1) or (P2) is true, the category $\mathcal{T}(X,x)$ is cofiltered; moreover in  $\mathcal{T}(X,x)$  has fiber products.
\end{thm}

\begin{rem}
\label{sharp_definition}
Let $V_1 = (G_1,T_1,t_1)$, $V_2 =(G_2,T_2,t_2)$, $V_{0}=(G_{0},T_{0},t_0)$ be objects in $T(X,x)$.  Let $(g_1,f_1)$ and $(g_2,f_2)$ be morphism from $(G_1,T_1,t_1)$ to $(G_0,T_0,t_0)$ and $(G_2,T_2,t_2)$ to $(G_0,T_0,t_0)$. Now consider the triple $W = (G_1 \times_{G_0} G_2, T_1 \times_{T_0} T_2, t_1 \times_{t_0} t_2)$. Note that the  $W$ and $V_1 \times_{V_0} V_2$ may not be the same. In fact, there is no 
guarantee that $W$ lies in the category $T(X,x)$ as $G_1  
\times_G G_2$ may not be flat over $S$. However, if $X$ is a 
scheme over a field $k$, then $W$ and $V_1 \times_{V} V_2$ are 
the same. The fiber product is obtained as follows. Let $Y$ be a
scheme over $S$. Let $\eta$ be the generic point of $S$.Then 
we define $Y^{\sharp}$ to be the schematic closure of 
$Y_{\eta} = Y \times_{S} \eta$ in $Y$. That is
$Y^{\sharp}$ is the closed subscheme 
$\spec(\mathcal{O}_{Y}/\text{tor})$ of $Y$ where $\text{tor}$ 
represents the torsion part as an $\mathcal{O}_S$ module . Then $V_1 \times_{V} V_2$
is $(G_3, T_3, t_3)$ where $G_3 = (G_1 \times_{G_0} 
G_2)^{\sharp}$ and $T_3 = (T_1 \times_{T_0} T_2)^{\sharp}$. 
\end{rem}

\section{Log Nori fundamental group scheme}
The existence of log Nori fundamental group scheme for log schemes over field $k$ was proved in \cite{sen2020log}. In this section we extend the definition to log schemes over Dedekind schemes. Let $S$ be a connected Dedekind scheme and $X$ be  a smooth connected scheme over $S$ . Let $D$ be a divisor with no multiplicity of $X$ such that the irreducible components of $D$ and as well their intersections are smooth over $S$. Now if we endow $X$ with the log structure associated with $D$ then the structure morphism from $X$ to $S$ is log-smooth. Let $x: S \to X$ be a $S$-point such that it doesn't intersect $D$. Let $\mathcal{T}^{\text{log}}(X,x)$ denote the following category: the objects of this category are triples of the form $(G,T,\alpha)$ where $T$ is a $G$-torsor on $X_{\text{fl}}^{\text{log}}$ site and $G$ is a classical finite flat group scheme over  $S$ with strict structure morphism and $\alpha$ is a trivialization $T|_{x} \cong G$. The morphisms between two objects in $\mathcal{T}^{\text{log}}(X,x)$, $( G_1, T_1, \alpha_1)$ and $(G_2, T_2, \alpha_2)$ are of the form $(g,f )$ where $g$ is group scheme homomorphism from $G_1$ to $G_2$ and $f$ is a morphism from $T_1$ to $T_2$ that respects the group action and the trivializations. A profinite group scheme $\pi^\text{log}_{\text{Nori}}(X,x)$ is called  log fundamental group scheme of $(X,x)$ if there is a triple $(\pi^\text{log}_{\text{Nori}}(X,x),T,\tau)$ such that $T$ is a $\pi^\text{log}_{\text{Nori}}(X,x)$-torsor over $X$ in the log flat topology and $\tau$ is trivialization and there exists a unique morphism from  $(\pi^\text{log}_{\text{Nori}}(X,x),T,\tau)$ to any object in $\mathcal{T}^{\text{log}}(X,x)$.

\begin{rem}
The torsors $T$ in the category  $\mathcal{T}^{\text{log}}(X,x)$ may not be representable and therefore we are looking at torsors which are actually sheaf where as the torsors in the category  $\mathcal{T}(X,x)$ (classical flat topology) are representable.
\end{rem}

Then we have the following theorem proved in this paper

\begin{thm}
\label{fgrpschm}
If $X$ satisfies the conditions mentioned above then $\pi^\text{log}_{\text{Nori}}(X,x)$ exists.
\end{thm}

We can define the \textit{log Nori fundamental group scheme} of $X$ as follows

$$ \pi^\text{log}_{\text{Nori}}(X,x) = \lim_{(G,T,\alpha) \in Ob(\mathcal{T}^{\text{log}}(X,x))} G.$$

Here limit is the projective limit of finite flat group schemes. In general, this limit may not exists. But it does exist if the index over which the limit is taken is small and cofiltered.

To prove that the category $\mathcal{T}^{\text{log}}(X,x)$ is cofiltered, we will show that the category $\mathcal{T}^{\text{log}}(X,x)$ has fiber products. 

Here is the brief outline of the proof. Consider three objects $V_1= ( G_1, Y_1, y_1)$ ,$V_2 = ( G_2, Y_2, y_2)$ and $V_0= ( G_0, Y_0, y_0)$ in $\mathcal{T}^{\text{log}}(X,x)$. Let $(g_1,f_1)$ and $(g_2,f_2)$ be morphism from $V_1$ to $V_0$ and $V_2$ to $V_0$. We find a Kummer log flat covering $U$ of $X$ such that the pullback of our torsors $Y_1$,$Y_2$ and $Y_0$ to $U$ are a classical torsors in the flat topology of $U$. If $U$ satisfies the property $P1$ then we can use \ref{torsorsfibreproductclassical} to compute the fiber product on $U$ and use descent of sheaves in log flat topology of $X$ to get a new torsor in log flat topology which is the fiber product of $V_1$,$V_2$ and $V_0$.

We first explain to how to construct the log flat covering $U$ described above. There exists some open covering of $(U_{\lambda})_{\lambda \in \Lambda}$ of $X$ and a morphism $g: U_{\lambda} \to \mathbb{A}_{S}^{r(\lambda)}$ given by a homomorphism of $\mathcal{O}_{S}$-algebras $\mathcal{O}_S[T_1, T_2, \ldots, T_{r(\lambda)}] \to g_{*} \mathcal{O}_{U_{\lambda}}$ where $T_i$ goes to $ t_i \in \Gamma(U_{\lambda}, \mathcal{O}_{U_{\lambda}})$ and $D \cap U_{\lambda} = \cup^{r(\lambda)}_{i=1}\{x \in {U_{\lambda}} : t_i(x)=0\}$.

For $n \geq 1$, consider $U_{\lambda,n } = \mathbb{A}_{S}^{r(\lambda)} \times_{\mathbb{A}_{S}^{r(\lambda)} } U_{\lambda}$, where the map $\mathbb{A}_{S}^{r(\lambda)} \to \mathbb{A}_{S}^{r(\lambda)}$ is given by the by a homomorphism of $\mathcal{O}_{S}$-algebras $\mathcal{O}_S[T_1, T_2, \ldots, T_{r(\lambda)}] \to  \mathcal{O}_S[T_1, T_2, \ldots, T_{r(\lambda)}]$ where $T_i \mapsto T_i^n$. The log structure on  $\mathbb{A}_{S}^{r(\lambda)}$ is given by $T_i$. The log structure on $U_{\lambda,n }$ is the inverse image of the log structure on $\mathbb{A}_{S}^{r(\lambda)}$ via the projection map of the fiber product. Note that $U_{\lambda,n }$ is smooth over $S$. Then we can define $U = \sqcup_{\lambda} U_{\lambda,n }$ for some $n \geq 1$.

\begin{thm}
\label{torsorsfibreproduct}
If $X$ satisfies the conditions mentioned above then the fiber product exists in $\mathcal{T}^{\text{log}}(X,x)$. 
\end{thm}

\begin{proof}
Let $Y$ be a scheme over $S$. Let $\eta$ be the generic point of $S$. Then we define $Y^{\sharp}$ to be the schematic closure of 
$Y_{\eta} = Y \times_{S} \eta$ in $Y$ (see 
\ref{sharp_definition}). Now consider three objects in $V_1= ( 
G_1, Y_1, y_1)$ ,$V_2 = ( G_2, Y_2, y_2)$ and $V_0= ( G_0, Y_0, 
y_0)$ in $\mathcal{T}^{\text{log}}(X,x)$. We will prove that the fiber product  $V_1 \times_{V_0} V_2$ exists in the category 
$\mathcal{T}^{\text{log}}(X,x)$. For some $n$, we can find a 
covering $U = \sqcup_{\lambda} U_{\lambda,n }$ such that the 
pullback $Y_i^{'}$ of $Y_i$ to $U$ is a classical $G_i$-torsor 
for $i \in \{0,1,2\}$. Now using Proposition 3.1 in 
\cite{gasbarri2020existence}, we see that the $Y_3^{'} = 
(Y_1^{'} \times _{{Y_0}^{'}} Y_2^{'})^{\sharp }$ is a classical 
$G_3$-torsor on $U$ where $G_3 = (G_1 \times_{G_0} 
G_2)^{\sharp}$. We now prove that $Y_3^{'}$ descends to a sheaf 
$Y_3$ on the Kummer log flat topology of $X$.
Consider the $i$-th projection maps $p_i : U\times_{X}U \to U$.
Let $Y_i^{''}$ represents the pullback of $Y_i$ to $ U 
\times_{X} U$ for $i \in \{0,1,2\}$. We have $p_1^{*} Y_3^{'} = 
(Y_1^{''} \times _{{Y_0}^{''}} Y_2^{''})^{\sharp } = p_2^{*} 
Y_3^{'}$ as the morphisms $p_i$ are flat. The isomorphism
$p_1^{*} Y_3^{'} \cong p_2^{*} Y_3^{'}$ is the descent datum of 
$Y_3^{'}$ and it descends to a unique sheaf $Y_3$ on $X$. Now, 
there is a natural action of $G_3$ on $Y_3$ and its pullback to 
$U$ which is log flat cover is a $G_3$-torsor. Therefore, $Y_3$ 
is a $G_3$-torsor.

Now, we explain how to get $y_3$. Recall we have a section $x: S \to X$. We have $Y_3 \subset Y_1 \times_{Y_0} Y_2$. Let $y_3$ be section of $x^{*}(Y_1 \times_{Y_0} Y_2)$ induced by $y_1$,$y_2$ and $y_0$. We show that $y_3$ belongs to $Y_3$. Let $S^{'} = S \times_{X} U$ then we have $x^{'}: S^{'} \to U$. We have a section $y_3^{'}$ of $(x^{'})^{*}(Y_1^{'} \times _{{Y_0}^{'}} Y_2^{'})$ induced by $y_3$. Now the $S$-morphism $y_3^{'}: S^{'} \to (Y_1^{'} \times _{{Y_0}^{'}} Y_2^{'})$ factors through $Y_3^{'} = (Y_1^{'} \times _{{Y_0}^{'}} Y_2^{'})^{\sharp }$ as $S'$ is flat over $S$. So, we have $y_3^{'}$ belongs to $(x^{'})^{*}(Y_3^{'} )$. Therefore, we have  $y_3$ belongs to $Y_3$.
\end{proof}

The Theorem \ref{fgrpschm} now follows from Theorem \ref{torsorsfibreproduct}.

\section{The log Picard scheme}
Let $S$ be a connected Dedekind Scheme and $X= (X,\mathcal{O}_X,M_X)$ a connected fine saturated log scheme over $S$ . Let $f$ be the structure morphism from $X$ to $S$. We also assume  $\mathcal{O}_S \simeq f_{*}(\mathcal{O}_X)$. Then we define $\mathbf{logPic}_{X/S}$ to be the sheaf  $R^1f_{*} \mathbb{G}_m$. Note that $\mathbb{G}_m$ is a sheaf on the log flat topology and $\mathbf{logPic}_{X/S}$ is sheaf on the classical flat topology of $U$. Now we can use the Theorem 

\begin{rem}
Note that the log Picard functor being described in this paper is different from the notion of usual log Picard functor described in \cite{olsson2004semistable} . The usual log Picard functor represents the group $H^1(X,M^{\text{gp}})$ where as our log Picard functor represents $H^1_{\text{log fl}}(X,\mathbb{G}_m)$.

\end{rem}

\begin{prop}
\label{picisomorphism}
Suppose there is a section from $s:S \to X$ to the structure morphism from $f:X \to S$. Then, for any finite flat commutative group scheme $G$ over $S$, we have $R^1f_{*}(G_X) \xrightarrow{\sim} \underline{\mathbf{Hom}_{S}}(G^{\vee},\mathbf{logPic}_{X/S})$, where $G^{\vee}$ is the Cartier dual of $G$.
\end{prop}
To prove the above theorem we will need the log version of the following theorem.
\begin{prop}
(Prop III.4.16 \cite{milne1980etale})
For any scheme $X$ and any finite flat commutative group scheme $G$ over $X$, $\underline{\mathbf{Ext}^{1}}_{X_{\text{fl}}} (G, \mathbb{G}_m) = 0$.
\end{prop}

The log version is as follows.

\begin{lem}
Let $X = (X,O_X,M_X)$ be a fine saturated log scheme and $G$ be a finite flat group scheme which we endow with a log structure from the inverse image of $M_X$. Then $\underline{\mathbf{Ext}^{1}}_{X^{\text{log}}_{\text{fl}}} (G, \mathbb{G}_m) = 0$.
\end{lem}
\begin{proof}
We have to show that for any $U$ in $X^{\log}_{\text{fl}}$ any element of     $\mathbf{Ext}^{1}_{U^{\text{log}}_{\text{fl}}}(G, \mathbb{G}_m)$ becomes trivial on some log flat cover of $U$. Every such extensions come from an exact sequence of sheaves
$$0 \to \mu_{n} \to E' \to G_{U} \to 0 $$ in log flat topology of $U$. Now $\text{id} \in G_{U}(G_{U})$. Then we can find a log flat covering $V \to N_{U}$ of $N_U$ where $\text{id} \in N_u(N_u)$ can be lifted to $E'(V)$. So, $E'$ splits  and it is classical torsor. Since, $\underline{\mathbf{Ext}^{1}}_{X_{\text{fl}}} (G, \mathbb{G}_m) = 0$ we can reduce it to the classical case.
\end{proof}
Consider a functor of sheaves $F$ on $X^{\text{log}}_{\text{fl}}$,
$$ H \mapsto F(H) = f_{*}{\underline{\mathbf{Hom}}(G^{\vee}_{X}, H) = \underline{\mathbf{Hom}}(G^{\vee},f_{*}(H))}$$.

Now take $H= \mathbb{G}_m$ and applying Grothendieck spectral sequence we get the following exact sequence

$$0 \to R^{1}f_{*}G_X \xrightarrow{\sim} R^{1}F(\mathbb{G}_m) \to f_{*}\underline{\mathbf{Ext}^{1}}_{X^{\text{log}}_{\text{fl}}}(G_X^{\vee}, \mathbb{G}_m) = 0.$$

This gives us the following exact sequence
 $$\underline{\mathbf{Ext}^{1}}(N^{\vee},\mathbb{G}_m) = 0 \to 
 R^1F(\mathbb{G}_m) \to {\underline{\mathbf{Hom}}}(N', 
 R^1f_{*} \mathbb{G}_m) \to 
 \underline{\mathbf{Ext}^{2}}(N^{\vee}, \mathbb{G}_m) \to R^2F(\mathbb{G}_m)$$
 
 We now need to prove that the map 
 $\underline{\mathbf{Ext}^{2}}(N^{\vee}, \mathbb{G}_m) \to R^2F(\mathbb{G}_m)$ is injective. Now we have a
 section  $s:S \to X$ to $f$. Therefore there is a morphism from 
 $f_{*}H \to s^{*}H$ which is functorial in $H$. Now this 
 morphism defines a map from $R^2F(H) \to 
 \underline{\mathbf{Ext}^{2}}(G^{\vee}, s^{*}H)$ whose 
 composition with $\underline{\mathbf{Ext}^{2}}(G^{\vee}, 
 f_{*}H) \to R^2F(H)$ is the map on 
 $\underline{\mathbf{Ext}^{2}}(G^{\vee},- )$ induced by $f_{*}H \to s^{*}H$. Now in our case the map $f_{*}\mathbb{G}_m \to s^{*}\mathbb{G}_m$ is an isomorphism so we get our desired result.
 
 \section{Abelian Fundamental Group Scheme}
 Let $X_{\text{fl}}$ denote the classical flat site on $X$ and $X^{\log}_{\text{fl}}$ denote the log flat site on $X$.
 
 Consider the following exact sequence on $X_{\text{fl}}$ site
 
 $$  0 \to \mathbb{G}_m \to M^{gp} \to M^{gp}/\mathbb{G}_m \to 0. $$.

 Similarly, consider the following exact sequence on $X^{\text{log}}_{\text{fl}}$ site
 
 $$  0 \to \mathbb{G}_m \to M^{gp} \to M^{gp}/\mathbb{G}_m \to 0 .$$.
 
 Now, let $f$ be the classical morphism  from $X$ to $S$ and $f^{\log}$ be the log morphism from $X$ to $S$. We can apply the functors   $f_{*}$  and $f^{\log}_{*}$ to both exact sequences above.
Note that   $f_{*}( M^{gp}/\mathbb{G}_m) =  \bigoplus_{1 \leq i \leq m} \mathbb{Z}_{D_i} $  where $D_i$ are the irreducible components of the divisor $D$. We also have $f^{\log}_{*}( M^{gp}/\mathbb{G}_m) =\bigoplus_{1 \leq i \leq m} \mathbb{Q}_{D_i} $. This gives us a map of the  complexes $$C_{1} = [\bigoplus_{1 \leq i \leq m} \mathbb{Z}_{D_i} \rightarrow \mathbf{Pic}_{X/S}]$$ and
 $$C_{2} = [\bigoplus_{1 \leq i \leq m} \mathbb{Q}_{D_i} \rightarrow \mathbf{logPic}_{X/S}]$$ on $S_{\text{fl}}$
 In the complex $C_1$, $\mathbf{Pic}_{X/S}$ is degree $0$ part and in the complex $C_2$, $\mathbf{logPic}_{X/S}$ is the degree $0$ part.
  Consider the following derived tensor product  $C_1 \otimes _ \mathbb{Z}^{\mathbf{L}} \mathbb{Z}/n\mathbb{Z}$ and  $C_2 \otimes _ \mathbb{Z}^{\mathbf{L}}  \mathbb{Z}/n\mathbb{Z}$.We have the following lemma
  
  \begin{lem}
  $H^{-1}(C_1 \otimes _ \mathbb{Z}^{\mathbf{L}} \mathbb{Z}/n\mathbb{Z}) \cong H^{-1}(C_2 \otimes _ \mathbb{Z}^{\mathbf{L}} \mathbb{Z}/n\mathbb{Z}).$
  \end{lem}
  
  \begin{proof} Note that there is an exact sequence of the form 
  $0 \to C_1 \to C_2.$ To prove the above the lemma it is enough to show that
  $\mathbf{logPic}_{X/S}/ \mathbf{Pic}_{X/S} = \bigoplus_{1 \leq i \leq m} \mathbb{Q}_{D_i}/\mathbb{Z}_{D_i}$. Let $\epsilon$ be the natural map of sites from $X^{\log}_{\text{fl}}$ to $X_{\text{fl}}$ induced by the functor that takes a classical scheme $Y$
over X and produces a log scheme with the underlying scheme  same as $Y$ and the log structure is induced by its structure
morphism. Now we have $f^{\log}_{*}= f_{*} \epsilon_{*}$. Using this we get the following exact sequence
$$0 \to  \mathbf{Pic}_{X/S} \to  \mathbf{logPic}_{X/S} \to  f_{*} (R^{1}\epsilon_{*}(\mathbb{G}_m) \to R^2f_{*}(\epsilon_{*}(\mathbb{G}_m)) .$$

Using Theorem 4.1 in \cite{kato2019logarithmic}, we get that $R^{1}\epsilon_{*}(\mathbb{G}_m)$ is $\bigoplus_{1 \leq i \leq m} \mathbb{Q}_{D_i}/\mathbb{Z}_{D_i}$. Therefore, we have an injection from $\mathbf{logPic}_{X/S}/ \mathbf{Pic}_{X/S} \to \mathbb{Q}_{D_i}/\mathbb{Z}_{D_i}$. On the other hand, from the complexes the $C_1$ and $C_2$ we get a map from $  \mathbb{Q}_{D_i}/\mathbb{Z}_{D_i} \to \mathbf{logPic}_{X/S}/ \mathbf{Pic}_{X/S}$. Composing the two above maps we get the identity map.
  \end{proof}
 
 Now, consider the following double complex

 \begin{tikzcd}
 (\bigoplus_{1 \leq i \leq m} \mathbb{Q}_{D_i}) \arrow[r] \arrow[d, "n" ] & \mathbf{logPic}_{X/S} \arrow[d, "n" ] \\
(\bigoplus_{1 \leq i \leq m} \mathbb{Q}_{D_i}) \arrow[r] &  \mathbf{logPic}_{X/S}\\
 \end{tikzcd}.

 The $H^{-1}$ of the total complex associated with the above double complex is equal to $H^{-1}(C_2 \otimes _ \mathbb{Z}^{\mathbf{L}} \mathbb{Z}/n\mathbb{Z})$ which is $\mathbf{logPic}_{X/S}][n]$. Then we get the following
 
 \begin{lem}
 $$H^{-1}(C_1 \otimes _ \mathbb{Z}^{\mathbf{L}} \mathbb{Z}/n\mathbb{Z}) \cong \mathbf{logPic}_{X/S}[n]$$
 \end{lem}

Using Theorem \ref{picisomorphism} for any finite flat group scheme we have an isomorphism 
$\omega : H^{1}_{\text{log fl}}(X,G)/H^1_{\text{fl}}(S,G) \xrightarrow[]{\sim} Hom_{S}(G^{\vee}, \mathbf{logPic}_{X/S} )$.

Let $H^{1}_{\bullet}(X,G) = H^{1}_{\text{log fl}}(X,G)/H^1(S,G)$ then we get the following

$$H^{1}_{\bullet}(X,G) \cong Hom_{S}(G^{\vee}, \mathbf{logPic}_{X/S} ).$$

Let $n$ be such that $nG=0$, then we get

$$H^{1}_{\bullet}(X,G) \cong Hom_{S}(G^{\vee}, \mathbf{logPic}_{X/S}[n] ) \cong Hom_{S}( (\mathbf{logPic}_{X/S}[n])^{\vee},G ).$$

Now from the definition of log Nori fundamental group we have

$$H^{1}_{\bullet}(X,G) \cong Hom ((\pi^\text{log}_{\text{Nori}}(X,x))^{\text{ab}},G)$$

Combining the last two equations we get.

\begin{prop}
\label{fundamental_pic}
\begin{equation*}
(\pi^\text{log}_{\text{Nori}}(X,x))^{\text{ab}} \cong \underset{n \in \mathbb{N}}{\varprojlim} (\mathbf{logPic}_{X/S}[n])^{\vee}    
\end{equation*}
\end{prop}
We also need the following lemma

\begin{lem}
\label{torsionses}
$(H^{-1}(C_0 \otimes _ \mathbb{Z}^{\mathbf{L}} \mathbb{Z}/n\mathbb{Z}) )^{\vee}  = \mathbf{Alb}_{X/S,D}[n]$
\end{lem}

\begin{lem}
We have an exact sequence of the form

$$0 \to H^{-1}(C_0 \otimes _ \mathbb{Z}^{\mathbf{L}} \mathbb{Z}/n\mathbb{Z}) \to  H^{-1}(C_1 \otimes _ \mathbb{Z}^{\mathbf{L}} \mathbb{Z}/n\mathbb{Z}) \to H^{-1}(\mathbf{NS}^{\tau}_{X/S,D} \otimes  _ \mathbb{Z}^{\mathbf{L}} \mathbb{Z}/n\mathbb{Z})\to 0$$
\end{lem}
\begin{proof}
We have an exact triangle of the form $$C_0 \otimes _ \mathbb{Z}^{\mathbf{L}} \mathbb{Z}/n\mathbb{Z} \to C_1 \otimes _ \mathbb{Z}^{\mathbf{L}} \mathbb{Z}/n\mathbb{Z} \to \mathbf{NS}^{\tau}_{X/S,D} \otimes  _ \mathbb{Z}^{\mathbf{L}} \mathbb{Z}/n\mathbb{Z} \to (C_0 \otimes _ \mathbb{Z}^{\mathbf{L}} \mathbb{Z}/n\mathbb{Z})[1]$$. Now, note that $H^{0}(C_0 \otimes _ \mathbb{Z}^{\mathbf{L}} \mathbb{Z}/n\mathbb{Z} ) =0$ and $H^{-2}(C_0, \mathbb{Z}/n\mathbb{Z}) =0$. Now this gives us an exact sequence of the form
$$0 \to H^{-1}(C_0 \otimes _ \mathbb{Z}^{\mathbf{L}} \mathbb{Z}/n\mathbb{Z}) \to  H^{-1}(C_1 \otimes _ \mathbb{Z}^{\mathbf{L}} \mathbb{Z}/n\mathbb{Z}) \to H^{-1}(\mathbf{NS}^{\tau}_{X/S,D} \otimes  _ \mathbb{Z}^{\mathbf{L}} \mathbb{Z}/n\mathbb{Z})\to 0$$

\end{proof}

\begin{thm}
We have an exact sequence of the form

$$0 \rightarrow{ \mathbf{NS}^{\tau}_{X/S,D}}^{\vee} \rightarrow (\pi^\text{log}_{\text{Nori}}(X,x))^{\text{ab}} \rightarrow \underset{n}{\varprojlim} \mathbf{Alb}_{X/S,D}[n] \rightarrow 0$$

where  $\mathbf{Alb}_{X/S,D} = \underline{\mathbf{Ext}}^{1}(C_{0}, \mathbb{G}_m)$.
\end{thm}
\begin{proof}
From lemma \ref{torsionses} we have 

$$0 \to H^{-1}(C_0 \otimes _ \mathbb{Z}^{\mathbf{L}} \mathbb{Z}/n\mathbb{Z}) \to  H^{-1}(C_1 \otimes _ \mathbb{Z}^{\mathbf{L}} \mathbb{Z}/n\mathbb{Z}) \to H^{-1}(\mathbf{NS}^{\tau}_{X/S,D} \otimes  _ \mathbb{Z}^{\mathbf{L}} \mathbb{Z}/n\mathbb{Z})\to 0$$

Now, $\mathbf{NS}^{\tau}_{X/S,D}[n] = H^{-1}(\mathbf{NS}^{\tau}_{X/S,D} \otimes  _ \mathbb{Z}^{\mathbf{L}} \mathbb{Z}/n\mathbb{Z}) $
Let us choose $n$ such that $n$ kills $ \mathbf{NS}^{\tau}_{X/S,D}$, then we have $\mathbf{NS}^{\tau}_{X/S,D}[n] = \mathbf{NS}^{\tau}_{X/S,D}$.

So, we have 
$$0 \to  H^{-1}(C_0 \otimes _ \mathbb{Z}^{\mathbf{L}} \mathbb{Z}/n\mathbb{Z}) \to  H^{-1}(C_1 \otimes _ \mathbb{Z}^{\mathbf{L}} \mathbb{Z}/n\mathbb{Z}) \to \mathbf{NS}^{\tau}_{X/S,D} \to 0$$

Now taking the Cartier dual we get

$$0 \to (\mathbf{NS}^{\tau}_{X/S,D})^{\vee}   \to (H^{-1}(C_1 \otimes _ \mathbb{Z}^{\mathbf{L}} \mathbb{Z}/n\mathbb{Z}))^{\vee} \to (H^{-1}(C_0 \otimes _ \mathbb{Z}^{\mathbf{L}} \mathbb{Z}/n\mathbb{Z}))^{\vee} \to  0$$

Now, $(H^{-1}(C_0 \otimes _ \mathbb{Z}^{\mathbf{L}} \mathbb{Z}/n\mathbb{Z}) )^{\vee}  = \mathbf{Alb}_{X/S,D}[n]$.

Now taking the inverse limit we get 
$$0 \rightarrow {\mathbf{NS}^{\tau}_{X/S,D}}^{\vee} \rightarrow (\pi^\text{log}_{\text{Nori}}(X,x))^{\text{ab}} \rightarrow \underset{n}{\varprojlim} \mathbf{Alb}_{X/S,D}[n] \rightarrow 0$$

\end{proof}
\section{Special Cases}
\subsection{Case of Curves}
If $X$ is a smooth proper curve over $S$ and $D$ is a divisor satisfying the conditions of Theorem \ref{mainthm} then the torsion subgroup scheme of the generalized Neron-Severi group scheme  $\mathbf{NS}^{\tau}_{X/S,D}$ is $0$. Also, in this case the  $\mathbf{Alb}_{X/S,D}$ is the generalized Jacobian $\mathbf{J}_{X/S,D}$. So this gives us the following theorem

\begin{thm}

Let $X$ be a proper smooth curve over $S$ and $D$ is a divisor
satisfying the conditions of Theorem \ref{mainthm} then we have
$$ (\pi^\text{log}_{\text{Nori}}(X,x))^{\text{ab}} \cong \underset{n}{\varprojlim} \mathbf{J}_{X/S,D}[n] $$ Here $\mathbf{J}_{X/S,D}$ is the generalized Jacobian of $X$.
\end{thm}

\subsection{ S is a $\mathbb{Q}$-scheme}
\label{char 0}
If $S$ is a Dedekind scheme over $\mathbb{Q}$ then we have a nice comparison theorem between the abelian etale fundamental group and abelian log Nori fundamental group scheme.

\begin{thm}
Let $X$ be a proper smooth scheme over $S$ which is a $\mathbb{Q}$-scheme  and $D$ is a divisor
satisfying the conditions of Theorem \ref{mainthm} then we have
$$ (\pi^\text{log}_{\text{Nori}}(X,x))^{\text{ab}} \cong  (\pi_{\text{et}}(U,x))^{\text{ab}} $$. Here $U = X-D$.
\end{thm}

\begin{proof}
We will show that the canonical map from $(\pi_{\text{et}}(U,x))^{\text{ab}}$ to $(\pi^\text{log}_{\text{Nori}}(X,x))^{\text{ab}}$ is an isomorphism. We know that $\operatorname{Hom}_{S}(\pi_{\text{et}}(U,x))^{\text{ab}},G) = H^1_{\text{et}}(U,G)$ and $\operatorname{Hom}_{S}(\pi^\text{log}_{\text{Nori}}(X,x)),G) = H^1_{\text{logfl}}(X,G)$ where $G$ is a finite abelian group scheme over $S$. If we work etale locally over $S$, then $G$ is isomorphic to the direct sum of $\mathbb{Z}/n\mathbb{Z}$.
Therefore it is enough to prove that the map from $H^1_{\text{logfl}}(X,\mathbb{Z}/n\mathbb{Z})$ to $H^1_{\text{et}}(U,\mathbb{Z}/n\mathbb{Z})$ is an isomorphism.
Let $\epsilon$ the canonical map of sites from $X_{\text{logfl}}$ to $X_{\text{fl}}$ and let $j : U \to X$. Then we have the following exact sequences

$$0 \to H^1_{\text{fl}}(X,\mathbb{Z}/n\mathbb{Z}) \to H^1_{\text{logfl}}(X,\mathbb{Z}/n\mathbb{Z}) \to H^0_{\text{fl}}(X_{\text{fl}}, R^1\epsilon_{*}\mathbb{Z}/n\mathbb{Z}) \to H^2(X_{\text{fl}},\mathbb{Z}/n\mathbb{Z})   $$

$$0 \to H^1_{\text{et}}(X,\mathbb{Z}/n\mathbb{Z}) \to H^1_{\text{et}}(U,\mathbb{Z}/n\mathbb{Z}) \to H^0_{\text{et}}(X_{\text{et}}, R^1j_{*}\mathbb{Z}/n\mathbb{Z}) \to H^2(X_{\text{et}},\mathbb{Z}/n\mathbb{Z})   $$

Now we have  $R^1j_{*}\mathbb{Z}/n\mathbb{Z}$ is isomorphic to  $R^1\epsilon_{*}\mathbb{Z}/n\mathbb{Z})$ which follows from Theorem 10.1 in \cite{nakayama2017logarithmic},  $H^1_{\text{et}}(X,\mathbb{Z}/n\mathbb{Z}) \cong  H^1_{\text{fl}}(X,\mathbb{Z}/n\mathbb{Z})$ and $H^2_{\text{fl}}(X,\mathbb{Z}/n\mathbb{Z}) \cong H^2_{\text{et}}(X,\mathbb{Z}/n\mathbb{Z})$. This gives us the map from  $H^1_{\text{logfl}}(X,\mathbb{Z}/n\mathbb{Z})$ to $H^1_{\text{et}}(U,\mathbb{Z}/n\mathbb{Z})$ is an isomorphism.
\end{proof}

\section{Abelian Fundamental Group Scheme of the fibers}
Let $S$ be a connected Dedekind scheme and $f:X \to S$ be a proper smooth connected $S$-scheme provided with a $S$-point $x: S \to X$. Then the size of the abelian Nori fundamental group scheme does not vary with fibers. More precisely, let $s \in S$ and consider the fiber of $f_s : X_s \to \spec(\kappa(s))$ and we have a section $x_s \in X_{s}(\spec(\kappa(s))$, then $(\pi_{\text{Nori}}(X,x))^{\text{ab}}_{s} \cong (\pi_{\text{Nori}}(X_s,x_s))^{\text{ab}}$. One can show a similar property holds for the log fundamental group scheme as well. Let $D$ be a divisor with no multiplicity of $X$ such that the irreducible components of $D$ as well as their intersections are smooth over $S$. Now we endow $X$ with the log structure associated with $D$. Then the structure morphism $f$ from $X$ to $S$ is log-smooth. We also assume our section $x$ does not intersect with $D$. Then the following property holds

\begin{lem}
\label{fibrefg}
$(\pi^{\text{log}}_{\text{Nori}}(X,x))^{\text{ab}}_{s} \cong (\pi^{\text{log}}_{\text{Nori}}(X_s,x_s))^{\text{ab}}.$
\end{lem}

\begin{proof}
Using Proposition \ref{fundamental_pic}, we have

$$(\pi^\text{log}_{\text{Nori}}(X,x))^{\text{ab}} \cong \underset{n \in \mathbb{N}}{\varprojlim} (\mathbf{logPic}_{X/S}[n])^{\vee}    $$. This should give us

$$(\pi^{\text{log}}_{\text{Nori}}(X,x))^{\text{ab}}_{s} = \underset{n \in \mathbb{N}}{\varprojlim} (\mathbf{logPic}_{ X_s/\spec(\kappa(s))}[n])^{\vee} = (\pi^{\text{log}}_{\text{Nori}}(X_s,x_s))^{\text{ab}}.$$

\end{proof}

The abelian log Nori Fundamental group scheme is especially well behaved when one moves from characteristic $0$ to characteristic $p>0$. Let $s,t \in S$ such that the characteristic of $\kappa(s)$ is $p>0$ and  the the characteristic of $\kappa(t)$ is $0$. Let $U_s = X_s - D_s$ and $U_t = X_t - D_t$ and the log structures on $X_s$ and $X_t$ are due to $D_s$ and $D_t$. Then  $(\pi^{\text{log}}_{\text{Nori}}(X_s,x_s))^{\text{ab}}$ and  $(\pi^{\text{log}}_{\text{Nori}}(X_t,x_t))^{\text{ab}}$ are fibers of the same finite flat group scheme by $\ref{fibrefg}$ where as the size of $(\pi^{et}(U_s,x_s))^{\text{ab}}$ is different from $(\pi^{et}(U_t,x_t))^{\text{ab}}$. Using \ref{char 0} we get that $(\pi^{et}(U_t,x_t))^{\text{ab}} \cong (\pi^{\text{log}}_{\text{Nori}}(X_t,x_t))^{\text{ab}}$. For instance, for the curve $X=\mathbb{P}_S-\{0,\infty\}$ in characteristic $0$, the $(\pi^{et}(X_t,x_t))^{\text{ab}}$ is $\underset{n}{\varprojlim} \mu_n$ and the abelian log Nori fundamental group scheme of $\mathbb{P}_S$ with $D = \{0,\infty\ \}$ is also $\underset{n}{\varprojlim} \mu_n$. In characteristic $p>0$ on the other hand, only the prime-to-$p$ part of $(\pi^{et}(X_s,x_s))^{\text{ab}}$ is $\underset{n}{\varprojlim} \mu_n$ and the total fundamental group is very large where as the abelian log Nori fundamental group scheme of $\mathbb{P}_S$ with $D = \{0,\infty\ \}$ is still $\underset{n}{\varprojlim} \mu_n$.
 
\bibliography{log_scheme}

\providecommand{\bysame}{\leavevmode\hbox to3em{\hrulefill}\thinspace}
\providecommand{\MR}{\relax\ifhmode\unskip\space\fi MR }
\providecommand{\MRhref}[2]{%
  \href{http://www.ams.org/mathscinet-getitem?mr=#1}{#2}
}
\providecommand{\href}[2]{#2}
\begin{thebibliography}{1}

\bibitem{antei2011abelian}
Marco Antei, \emph{On the abelian fundamental group scheme of a family of
  varieties}, Israel Journal of Mathematics \textbf{186} (2011), no.~1,
  427--446.

\bibitem{gasbarri2020existence}
Carlo Gasbarri, Michel Emsalem, and Marco Antei, \emph{Sur l'existence du
  sch{\'e}ma en groupes fondametal}, {\'E}pijournal de G{\'e}om{\'e}trie
  Alg{\'e}brique \textbf{4} (2020).

\bibitem{kato2019logarithmic}
Kazuya Kato, \emph{Logarithmic structures of fontaine-illusie. ii}, arXiv
  preprint arXiv:1905.10678 (2019).

\bibitem{milne1980etale}
James~S Milne, \emph{Etale cohomology (pms-33)}, vol. 5657, Princeton
  university press, 1980.

\bibitem{nakayama2017logarithmic}
Chikara Nakayama, \emph{Logarithmic {\'e}tale cohomology, ii}, Advances in
  Mathematics \textbf{314} (2017), 663--725.

\bibitem{olsson2004semistable}
Martin~C Olsson et~al., \emph{Semistable degenerations and period spaces for
  polarized k3 surfaces}, Duke Mathematical Journal \textbf{125} (2004), no.~1,
  121--203.

\bibitem{sen2020log}
Aritra Sen, \emph{Log fundamental group scheme}, 2020.

\end{thebibliography}
\bibliographystyle{amsplain}

\end{document}